\documentclass{amsart}

\usepackage{amsfonts,amsmath,amsthm}
\usepackage{enumitem} 

\usepackage{mathtools}    
\DeclarePairedDelimiterX\setc[2]{\{}{\}}{\,#1 \;\delimsize\vert\; #2\,}

\usepackage[dvipsnames]{xcolor}

\newcommand{\R}{\mathbb{R}}

\newcommand{\B}{\mathbb{B}}

\renewcommand{\S}{\mathbb{S}}

\newcommand{\Sn}{\mathbb{S}^n}

\renewcommand{\O}{\Omega}
\renewcommand{\AA}{|\Omega|}

\newtheorem{defi}{Definition}[section] 
\newtheorem{thm}[defi]{Theorem}

\newtheorem{rem}[defi]{Remark}
 
 \newtheorem{lemma}[defi]{Lemma}
\newtheorem{cor}[defi]{Corollary}

\title[Lower bounds for the sum of the reciprocals of eigenvalues]{Lower bounds for the sum of the reciprocals of eigenvalues of bounded domains in \( \mathbb{R}^n \), spheres, and closed orientable surfaces
}
\author{Mehdi Eddaoudi}

\begin{document}
\begin{abstract}
We establish lower bounds for the sum of the reciprocals of eigenvalues of the Laplacian. For bounded domains, our result extends the upper bound provided by Bucur and Henrot on the second Neumann eigenvalue and is related to a result by Wang and Xia, which connects to a conjecture of Ashbaugh and Benguria. For spheres and surfaces, we extend known results on the first and second eigenvalues, and strengthen an analogous conjecture involving the conformal volume of Li and Yau.
\end{abstract}
\maketitle


\section{Introduction and Main Results}

\subsection{Neumann eigenvalues of bounded domain in $\R^n$.}
Let \( \Omega \) be a regular, bounded domain in \( \mathbb{R}^n \), and consider the classical Neumann eigenvalue problem:
\[
\begin{cases}
\Delta f = -\mu f & \text{in } \Omega, \\
\frac{\partial f}{\partial \nu} = 0 & \text{on } \partial \Omega,
\end{cases}
\]
where \( \frac{\partial}{\partial \nu} \) represents the outward unit normal derivative on the boundary \( \partial \Omega \). It is well known that the spectrum of this problem is real, discrete, and consists of an infinite sequence of eigenvalues given by
\[
0 = \mu_0(\Omega) < \mu_1(\Omega) \leq \mu_2(\Omega) \leq \cdots \to +\infty.
\]
For each eigenvalue \( \mu_j(\Omega) \), let \( \{f_j\}_{j \geq 0} \) denote an orthonormal basis of \( L^2(\Omega) \) formed by the corresponding eigenfunctions. The eigenvalues are characterized variationally as follows:
\begin{equation}\label{caractérisation variationnelle Domaine}
\mu_k(\Omega) = \min_{f \in A_k \setminus \{0\}} \frac{\int_\Omega |\nabla f|^2 \, dx}{\int_\Omega f^2 \, dx},
\end{equation}
where \( A_k \) is the subspace of \( H^1(\Omega) \) defined by the orthogonality condition:
\[
A_k = \left\{ f \in H^1(\Omega) \mid \int_\Omega f f_j \, dx = 0 \text{ for } j = 0, 1, \dots, k-1 \right\}.
\]

In the two-dimensional setting, Szeg\H{o}~\cite{Szego} established that if \( \Omega \subset \mathbb{R}^2 \) is a simply connected bounded planar domain, then the first non-trivial eigenvalue satisfies the following sharp inequality
\begin{equation}
\mu_1(\Omega) \AA \leq \mu_1(\mathbb{D}) |\mathbb{D}| \sim 3.39 \pi,
\end{equation}
where \( \AA \) denotes the area of \( \Omega \), and \( \mathbb{D} \) is the unit disk. Later, Weinberger~\cite{Weinberger} extended this result to higher dimensions, proving that for any bounded domain \( \Omega \subset \mathbb{R}^n \)
\begin{equation}
\mu_1(\Omega) |\Omega|^{2/n} \leq \mu_1(\mathbb{B}^n) |\mathbb{B}^n|^{2/n},
\end{equation}
where \( \mathbb{B}^n \) denotes the unit ball in \( \mathbb{R}^n \). Szeg\H{o} and Weinberger also observed that their proof in the planar case could be extended to derive the following bound for the reciprocals of the first two eigenvalues
\begin{equation}\label{eq:AshbaughBenguria}
\frac{1}{\mu_1(\Omega)} + \frac{1}{\mu_2(\Omega)} \geq \frac{2 \AA}{\pi \mu_1(\mathbb{D})}.
\end{equation}
These bounds are sharp, with equality attained if and only if \( \Omega \) is a disk (or an \( n \)-dimensional ball in higher dimensions).

The upper bound on the second Neumann eigenvalue was established by Girouard, Nadirashvili, and Polterovich~\cite{GNP}, who introduced a method involving folding measures into hyperbolic caps within the Poincaré disk model. They proved that for any simply connected bounded domain \( \Omega \subset \mathbb{R}^2 \), the second eigenvalue satisfies
\begin{equation}\label{eq:GNP}
\mu_2(\Omega) |\Omega| < \mu_2(\mathbb{D} \sqcup \mathbb{D}) |\mathbb{D} \sqcup \mathbb{D}|,
\end{equation}
where \( \mathbb{D} \sqcup \mathbb{D} \) denotes the disjoint union of two identical disks.

Building upon this folding approach in conjunction with Weinberger’s argument, Bucur and Henrot~\cite{BucurHenrot} extended the result to higher dimensions. For any bounded domain \( \Omega \subset \mathbb{R}^n \), then
\begin{equation}\label{eq:BucurHenrot}
\mu_2(\Omega) |\Omega|^{2/n} < \mu_2(\mathbb{B}^n \sqcup \mathbb{B}^n) |\mathbb{B}^n \sqcup \mathbb{B}^n|^{2/n}.
\end{equation}

Equality in both inequalities \eqref{eq:GNP} and \eqref{eq:BucurHenrot} is attained in the limit where the domain splits into two identical disjoint disks (or two disjoint \( n \)-dimensional balls in the higher-dimensional case).

For higher-dimensional bounded domains in \( \mathbb{R}^n \), the analogue of inequality \eqref{eq:AshbaughBenguria} is a conjecture by Ashbaugh and Benguria~\cite{AshBenConjec}. It states that
\begin{equation}\label{conjABDirichlet}
\frac{1}{\mu_1(\Omega)} + \cdots + \frac{1}{\mu_n(\Omega)} \geq \frac{n|\Omega|^{2/n}}{\mu_1(\mathbb{B}^n) |\mathbb{B}^n|^{2/n}},
\end{equation}
with equality holding if and only if \( \Omega \) is a ball.

Using symmetrization arguments from Chiti’s work~\cite{chiti1983bound}, Ashbaugh and Benguria~\cite{AshBenConjec} proved this conjecture up to a dimensional constant, namely that
\begin{equation}\label{Ineq:AshBeng}
\frac{1}{\mu_1(\Omega)} + \cdots + \frac{1}{\mu_n(\Omega)} \geq \frac{1}{n+2} \frac{n|\Omega|^{2/n}}{\mu_1(\mathbb{B}^n) |\mathbb{B}^n|^{2/n}}.
\end{equation}

The best known result supporting this conjecture to date is due to Wang and Xia \cite{Xia2018OnAC},
\begin{equation}\label{eq:Wang Xia Neumann}
\frac{1}{\mu_1(\Omega)} + \cdots + \frac{1}{\mu_{n-1}(\Omega)} \geq \frac{(n-1)|\Omega|^{2/n}}{\mu_1(\mathbb{B}^n) |\mathbb{B}^n|^{2/n}},
\end{equation}
with equality holding if and only if \( \Omega \) is a ball in \( \mathbb{R}^n \).

Similar results and conjectures also exist in the context of domains in a sphere or in a curved space \cite{bucur2022sharpsphere, benguria2020sharp,chen2024upper,langford2023maximizers}.

\medskip

Starting from the sum in inequality \eqref{eq:Wang Xia Neumann} from \( 1 / \mu_2 \), we establish the following sharp lower bound.

\begin{thm}\label{thmprncpalDomaine}
Let \( \Omega \subset \mathbb{R}^n \) be a regular bounded domain in \( \mathbb{R}^n \). Then 
\[
\frac{1}{\mu_2(\Omega)} + \cdots + \frac{1}{\mu_{n}(\Omega)} > \frac{(n-1)|\Omega|^{2/n}}{\mu_2(\mathbb{B}^n \sqcup \mathbb{B}^n) |\mathbb{B}^n \sqcup \mathbb{B}^n|^{2/n}},
\]
with equality achieved by a sequence of 
domains splitting into two balls of identical volume.
\end{thm}

Theorem~\ref{thmprncpalDomaine} provides a natural strengthening of Bucur and Henrot's inequality~\eqref{eq:BucurHenrot}. It is worth mentioning that Bucur, Martinet, and Nahon~\cite{bucur2022sharpsphere} established the same inequality when \( \Omega \subset \mathbb{S}^n \) is an open Lipschitz set. Moreover, in Remark~10 of the same work, they already observed that their result should hold in the Euclidean setting as well, although without providing a proof.  

An immediate consequence of Theorem \ref{thmprncpalDomaine} is an improvement of Ashbaugh and Benguria's inequality~\eqref{Ineq:AshBeng}.

\begin{cor}
Let \( \Omega \subset \mathbb{R}^n \) be a regular bounded domain in \( \mathbb{R}^n \). Then 
\[
\frac{1}{\mu_1(\Omega)} + \cdots + \frac{1}{\mu_{n}(\Omega)} > \frac{\left( (n-1)/2^{2/n} + 1 \right)|\Omega|^{2/n}}{\mu_1(\mathbb{B}^n) |\mathbb{B}^n|^{2/n}}.
\]
\end{cor}

\subsection{Eigenvalue estimates on closed manifolds}

Let \( (M,g) \) be a closed Riemannian manifold of dimension \( n \), and let 
\[
\lambda_0(M,g) = 0 < \lambda_1(M,g) \leq \cdots \to +\infty
\]
be the spectrum of the Laplace-Beltrami operator \( \Delta_g \).

\smallskip

Inspired by Szeg\H{o}'s work \cite{Szego}, Hersch \cite{Hersch} used conformal automorphisms of the $2$-sphere $\S^2$ to prove that for any metric $g$ on the sphere $\S^2$
\begin{equation}\label{Herschsomme}
     \frac{1}{\lambda_1(\mathbb{S}^2, g)} + \frac{1}{\lambda_2(\mathbb{S}^2, g)} + \frac{1}{\lambda_3(\mathbb{S}^2, g)} \geq \frac{3 \text{Area}(g)}{8 \pi},
\end{equation}
where $g_0$ is the standard round metric on the sphere. Hersch's inequality \eqref{Herschsomme} was later extended by Yang and Yau \cite{YangYau} to closed orientable surfaces \((M, g)\) of genus \(\gamma\), and further refined by El Soufi and Ilias \cite{ElSoufiIlias1984}, leading to 
\begin{equation}\label{YangYausomme}
     \frac{1}{\lambda_1(M, g)} + \frac{1}{\lambda_2(M, g)} + \frac{1}{\lambda_3(M, g)} \geq \frac{3 \text{Area}(g)}{8 \pi \left\lfloor \frac{\gamma+3}{2} \right\rfloor},
\end{equation}
where \(\left\lfloor \cdot \right\rfloor\) denotes the floor function.

Inequalities \(\eqref{Herschsomme}\) and \(\eqref{YangYausomme}\) naturally imply a topological bound for the first eigenvalue of orientable surfaces
\begin{equation}\label{YangYauestimation}
    \lambda_1(M, g) \text{Area}(g) \leq 8 \pi \left\lfloor \frac{\gamma+3}{2} \right\rfloor.
\end{equation}
Moreover, through the recent porgress of Karpukhin, Nadirashvili, Penskoi, and Polterovich~\cite{KNPP2021,KNPP2}, we know that 
\begin{gather}\label{ineq:KNPPSurfaces}
    \lambda_k(M,g)\text{Area}(g)  \leq 8 \pi k \left\lfloor \frac{\gamma+3}{2} \right\rfloor
\end{gather}
holds for an arbitrary \(k \geq 1\).

In contrast, for higher-dimensional closed manifolds, the situation is quite different, as no upper bounds can be established solely in terms of topology. Instead, El Soufi and Ilias~\cite{el1986immersions} expressed an upper bound on the first eigenvalue using a conformal invariant: the conformal volume. Specifically, for each conformal class \(C\) on \(M\), and for any Riemannian metric \(g \in C\) that admits a conformal immersion \(\phi: (M, C) \to \mathbb{S}^m \subset \mathbb{R}^{m+1}\), then
\begin{gather}\label{ineq:ElSoufiIlias}
    \lambda_1(M, g)\text{vol}(M,g)^{2/n} \leq n V_c(m, M, C)^{2/n},
\end{gather}
where \(V_c(m, M, C)\) denotes the $m$-conformal volume.

Equality occurs if and only if there exists a minimal immersion \(\phi: M \to \mathbb{S}^m\) satisfying
\begin{gather}\label{eq:eigenphi}
 -\Delta_g \phi = \lambda_1(M, g) \phi,
\end{gather}
and \(\phi^{*}g_{\S^m} = k g\) for some constant \(k > 0\).

In a similar vein to the Neumann case, the reciprocal sum associated with inequality \eqref{YangYausomme} remains an open question in higher dimensions, and we conjecture that the following is expected to be true
\begin{equation}\label{conjecturevariété}
     \frac{1}{\lambda_1(M,g)} + \cdots + \frac{1}{\lambda_{m+1}(M,g)} \geq \frac{(m+1) \, \text{vol}(M,g)^{2/n}}{ n V_c(m,M,C)^{2/n}}.
\end{equation}
A first step in this direction is a result on the \(n\)-sphere \(\S^n\). Denote by $g_0$ its standard round metric and by $w_n$ its volume. We prove that:
\begin{thm}\label{thmprincSphère1}
    For any metric \(g\) on \(\mathbb{S}^n\) that is conformal to the standard metric \(g_0\), we have
    \begin{equation}
    \frac{1}{\lambda_1(\mathbb{S}^n, g)} + \cdots + \frac{1}{\lambda_{n+1}(\mathbb{S}^n, g)} \geq \frac{(n+1) \, \text{vol}(\S^n,g)^{2/n}}{n K_n w_n^{2/n}},
    \end{equation}
    where \(K_n\) is a dimensional constant defined as
    \begin{equation}\label{defi de Kn}
    K_n = \frac{n+1}{n} \left( \frac{\Gamma(n)\Gamma\left(\frac{n+1}{2}\right)}{\Gamma\left(n+\frac{1}{2}\right)\Gamma\left(\frac{n}{2}\right)} \right)^{2/n},
    \end{equation}
    and where \(\Gamma\) is the Gamma function.
\end{thm}
In particular, the constant \(K_n\) satisfies the following properties:
\[
K_2 = 1, \quad 1 \leq K_n \leq 1.04 \quad \text{for \(n \geq 3\)}, \quad \text{and} \quad \lim_{n \to \infty} K_n = 1.
\]

It is important to highlight that the conformal volume can be computed in various situations by considering the equality case in \eqref{ineq:ElSoufiIlias}, with explicit examples provided in \cite{el1986immersions,EddaoudiGirouard}. In the particular case of the sphere \(\mathbb{S}^n\), we have \(V_c(n, \mathbb{S}^n, [g_0]) = w_n\), therefore the conjecture \eqref{conjecturevariété} holds up to the constant \(K_n\).

This constant was first computed by Girouard, Nadirashvili, and Polterovich in \cite{GNP}, where they proved that for an odd-dimensional sphere \(\mathbb{S}^n\) endowed with a metric \(g\) conformal to the standard metric \(g_0\), then
\begin{equation}\label{GNPsphere}
    \lambda_2(\mathbb{S}^n, g) \, \text{vol}(\mathbb{S}^n, g)^{2/n} < 2^{2/n} n K_n w_n^{2/n}.
\end{equation}
They also conjectured in the same paper that this bound should hold without the constant \(K_n\), with equality occurring when the metric \(g\) degenerates into two round spheres of equal volume. This phenomenon, often referred to as bubbling, traces back to Nadirashvili \cite{nadirashvili2002isoperimetric} who established the optimal result on the $2$-sphere. Note that this phenomenon is also present in Bucur and Henrot's inequality \eqref{eq:BucurHenrot}. For a more detailed discussion, see \cite{PetridesExistence,petridesregularite,KNPP2021,KNPP2}. In other words, they conjectured that
\begin{equation}\label{GNPsphereconj}
    \lambda_2(\mathbb{S}^n, g) \, \text{vol}(\mathbb{S}^n, g)^{2/n} < 2^{2/n} n w_n^{2/n}.
\end{equation}
Over the last two decades, various contributions have led to the eventual proof of this result. Notably, Petrides \cite{petrides2014maximization} addressed the case for even-dimensional spheres, Freitas and Laugesen \cite{F-L1} improved the topological argument of symmetry used by Petrides, and finally, Kim \cite{kim2022maximization} obtained the optimal result. More recently, the author and Girouard \cite{EddaoudiGirouard} derived a general upper bound for closed manifolds in the spirit of El Soufi and Ilias's inequality \eqref{ineq:ElSoufiIlias}.

\smallskip

By considering the sum of reciprocal eigenvalues starting from \( 1/\lambda_2(M, g) \), we prove the following.

\begin{thm}\label{thmprincsurfaces}
    Let \((M,g)\) be a closed orientable surface of genus \(\gamma\), then the following inequality holds:
    \begin{equation}
    \frac{1}{\lambda_2(M,g)} + \frac{1}{\lambda_3(M,g)} + \frac{1}{\lambda_4(M,g)} \geq \frac{3 \text{Area}(g)}{16 \pi \left\lfloor \frac{\gamma+3}{2} \right\rfloor}.
    \end{equation}
\end{thm}

\begin{thm}\label{thmprincsphere2}
    Let \(g\) be a metric on \(\S^n\) conformal to the standard metric \(g_0\), then the following inequality holds:
    \begin{equation}
    \frac{1}{\lambda_2(\mathbb{S}^n, g)} + \cdots + \frac{1}{\lambda_{n+2}(\mathbb{S}^n, g)} > \frac{(n+1) \, \text{vol}(\S^n, g)^{2/n}}{2^{2/n} n K_n w_n^{2/n}}.
    \end{equation}
\end{thm}

\medskip

\subsection*{Plan of the Paper} 
The paper is organized as follows:

In Section 2, we first review Weinberger's center of mass argument for domains in \( \mathbb{R}^n \), followed by the folding method of Bucur and Henrot. We then examine the work of Ashbaugh and Benguria, discussing its compatibility with this folding construction, and then we conclude with the proof of Theorem \ref{thmprncpalDomaine}.

In Section 3 we deal with spheres and surfaces. After recalling the Hersch-Szeg\H{o} center of mass and the folding into spherical caps, we prove Theorem \ref{thmprincSphère1} and Theorem \ref{thmprincsphere2}. Then we conclude the section by proving Theorem \ref{thmprincsurfaces} for surfaces.
\medskip



\section{The center of mass and folding via Bucur and Henrot's approach}
Let \(\Omega\) be a bounded domain in \(\mathbb{R}^n\). Throughout the paper, we define \( R_\Omega \) as the radius of a ball with the same volume as \(\Omega\), \( r_\Omega \) as the radius of a ball whose volume is \(\frac{|\Omega|}{2}\), and \( B_R \) as the ball of radius \( R \) centered at the origin.

\subsection{Center of mass argument: the Weinberger way}
Let $A$ a point in $\R^n$, and $d_A(x)$ the distance function from $x$ to $A$. For any $R >0$, Weinberger \cite{Weinberger} introduced the following function $g_A : \R^n \to \R^n$ defined as:
\begin{equation}\label{FonctionWeinberger}
    g_A(x) = \frac{G_R(d_A(x))}{d_A(x)} \overrightarrow{Ax},
\end{equation}
with  \(G_R : [0, +\infty) \to \mathbb{R}\) is the function defined as
\[
G_R(t) = 
\begin{cases}
g(t), & t \leq R, \\
g(R), & t \geq R,
\end{cases}
\]
where \(g\) is a radial function associated with the Bessel function of the first kind
$$g(|x|) = |x|^{1 - \frac{n}{2}} J_{n/2}\left(\sqrt{\mu_1(B_R)}|x|\right).$$
In particular, the function \(g\) satisfies the differential equation
\[
g''(t) + \frac{n-1}{t} g'(t) + \left( \mu_1(B_R) - \frac{n-1}{t^2} \right) g(t) = 0,
\]
with boundary conditions
\[
g(0) = 0, \quad g'(R) = 0. 
\]
 Using Brouwer fixed theorem, Weinberger \cite{Weinberger} proved that for $R= R_\Omega$, there exists a point $A \in \R^n$ such that for any orthonormal basis \(e_i\) of $\R^n$
\[
\int_\O g_A(x) \cdot e_i \, dx=0.
\]
Therefore, by the variational characterization of the first Neumann eigenvalue $\mu_1(\Omega)$, after summing up on the components of $g_A$, he obtained
\[
\mu_1(\Omega) \leq 
\frac{\int_\Omega G'_{R_\Omega}(d_A(x))^2 + (N-1) \frac{G_{R_\Omega}(d_A(x))^2}{d_A(x)^2} \, dx}
{\int_\Omega G_{R_\Omega}(d_A(x))^2 \, dx}.
\]
Weinberger's reasoning proceeds by comparing the right-hand side to that of the ball \( B_{R_\Omega} \). This approach is known as \textit{mass displacement}. Finally, he proved that
\[
\mu_1(\Omega) \leq \mu_1(B_{A,R_\Omega}),
\]
where \( B_{A,R_\Omega} \) denotes the ball centered at \( A \) with radius \( R_\Omega \).

\subsection{Folding through hyperplanes}
In their work \cite{BucurHenrot}, Bucur and Henrot extended Weinberger's construction by folding the function $g_A$, thereby allowing it to satisfy two orthogonality conditions simultaneously.
 This mechanism has been well-established in various situations, including spheres, simply connected planar domains and hyperbolic domains, see \cite{nadirashvili2002isoperimetric, GNP, petrides2014maximization, kim2022maximization, kim2024second, F-L1, GP, GL}.  Brief summary of this construction is provided below.
\smallskip

Let \( A \) and \( B \) be two distinct points in \( \mathbb{R}^n \). We define the linear part of the symmetry operator 
with respect to the hyperplane \( H_{AB} \), which is the mediator of the segment \( AB \), as
\[
T_{AB} : \mathbb{R}^n \to \mathbb{R}^n, \quad T_{AB}(v) = v - 2 (\vec{ab} \cdot v) \vec{ab},
\]
where 
\[
\vec{ab} = \frac{\overrightarrow{AB}}{\|\overrightarrow{AB}\|}.
\]
The hyperplane \( H_{AB} \) divides \( \mathbb{R}^n \) into two half-spaces: \( H_A \), containing \( A \), and \( H_B \), containing \( B \). Using these, we define the function $g_{AB} : \mathbb{R}^n \to \mathbb{R}^n$ by
\[
g_{AB}(x) =
\begin{cases} 
g_A(x), & \text{if } x \in H_A, \\
T_{AB}(g_B(x)), & \text{if } x \in H_B.
\end{cases}
\]
Here, \( g_A \) and \( g_B \) are the Weinberger functions from equation \eqref{FonctionWeinberger}, based on \( G_{r_\Omega} \). To be precise, on each of the half-spaces \( H_A \) and \( H_B \), the restriction of \( g_{AB} \) acts like a Weinberger function associated with a ball of half measure. Moreover, it also remains continuous along the hyperplane \( H_{AB} \).

Denoting by \( f_1 \) the first non-constant eigenfunction of \( \Omega \), Bucur and Henrot \cite{BucurHenrot}[Proposition 7] proved that for \( R = r_\Omega \), there exist \( A, B \in \mathbb{R}^n \) such that
\begin{equation}\label{FonctionBucuHenrot}
    \forall i = 1, \dots, n, \quad
\int_\Omega g_{AB} \cdot e_i \, dx = 
\int_\Omega g_{AB} \cdot e_i f_1 \, dx = 0.
\end{equation}
Therefore, by the variational characterization of $\mu_2(\Omega)$, this construction ensures that:
\begin{equation}
    \mu_2(\Omega) \leq 
\frac{\sum_{i=1}^n \int_\Omega |\nabla (g_{AB} \cdot e_i)|^2 \, dx}
{\sum_{i=1}^n \int_\Omega |g_{AB} \cdot e_i|^2 \, dx}.
\end{equation}
Splitting these two integrals over \( \Omega \cap H_A \) and \( \Omega \cap H_B \), and using a mass displacement argument on each of these two sets as done in the proof by Weinberger, they obtained 
\begin{equation}\label{ineq: strictBH}
    \mu_2(\Omega) \leq \mu_1(B_{r_\Omega}). 
\end{equation}
where  $\mu_1(B_{r_\Omega})$ is the first eigenvalue of a ball having a volume of \( \frac{|\Omega|}{2} \).

Since this construction guarantees that equality can only occur when the domain consists of two balls of equal volume, inequality \eqref{ineq: strictBH} is always strict.

\subsection{Compatibility of Ashbaugh and Benguria's topological argument with the folding method}
Building on the Weinberger's center of mass argument and using Borsuk-Ulam theorem, Ashbaugh and Benguria \cite{AshBenConjec} constructed a family of test functions for \( \mu_i(\Omega) \), where \( i = 1, \dots, n \). These functions are essentially components of the coordinate functions of \( \mathbb{R}^n \), but with a careful selection of an appropriate basis for projection. They then proceeded by estimating the Rayleigh quotient using a symmetrization as in the work of Chiti \cite{chiti1983bound}.

In a more recent development, Wang and Xia \cite{Xia2018OnAC} sought to refine Ashbaugh and Benguria's approach by considering the components of the function \( g_A \) as the test functions. While they didn’t manage to prove the full conjecture of Ashbaugh and Benguria, their work resulted in a significant progress, leading to an optimal bound for the $n-1$ first terms, see inequality \eqref{eq:Wang Xia Neumann}.

One might emphasize that the use of topological arguments to construct test functions is a well-known approach that can yield optimal bounds for eigenvalues. One of the key contribution of this paper is that the folding argument developed by Bucur and Henrot, which is also topological in nature, integrates seamlessly with the strategy employed by Ashbaugh and Benguria. This compatibility enables us to use the components of \( g_{AB} \) as test functions for \( \mu_{i+1}(\Omega) \), where \( i = 1, \dots, n \). The proof of Theorem \ref{thmprncpalDomaine} then proceeds in a manner similar to the Rayleigh quotient estimates developed by Wang and Xia.
\smallskip

Let \( f_i \) be the eigenfunction associated with the Neumann eigenvalue \( \mu_i \) on \( \Omega \), as defined in the introduction, and let $A$ and $B$ two point in $\R^n$ such that \( g_{AB} \) satisfies inequality \eqref{FonctionBucuHenrot}. In the following lemma, we show that the components of \( g_{AB} \)  can be chosen as test functions for \( \mu_{i+1}(\Omega) \), where \( i = 1, \dots, n \).

\begin{lemma}\label{Topological arguments BU domains}
Let \( \Omega \) be a regular bounded domain in \( \mathbb{R}^n \). Then, there exists an orthonormal basis \( (e_i) \) of \( \mathbb{R}^n \) such that for \( i = 1, \dots, n \), we have:
\begin{equation}
  \mu_{i+1}(\Omega) \int_\Omega |g_{AB}(x) \cdot e_i|^2 \, dx \leq \int_\Omega |\nabla g_{AB}(x) \cdot e_i|^2 \, dx.
\end{equation}
\end{lemma}

 \begin{proof}
     
Consider the map \( F_{n} : \mathbb{S}^{n-1} \to \mathbb{R}^{n-1} \) defined as
\[
F_{n}(p) =
\begin{bmatrix} 
\int_{\Omega} g_{AB}(x) \cdot p \times f_2(x) \, dx \\
\vdots \\
\int_{\Omega} g_{AB}(x) \cdot p f_{n}(x) \, dx
\end{bmatrix}.
\]

Is is straight forward to see that $F_n$ is an odd and continuous map, so by the Borsuk-Ulam theorem, there exists a vector \( e_{n} \) such that $$ F_{n}(e_{n}) = 0.$$ 
Next, let \( H \) be the hyperplane in \( \mathbb{R}^{n} \) orthogonal to \( e_{n} \), and consider the map \( F_{n-1} \) defined on the unit sphere \( \mathbb{S}^{n-2} \subset H \) with values in \( \mathbb{R}^{n-2} \) by
\[
F_{n-1}(p) =
\begin{bmatrix} 
\int_{\Omega} g_{AB}(x) \cdot p f_2(x) \, dx \\
\vdots \\
\int_{\Omega} g_{AB}(x) \cdot p f_{n-1}(x) \, dx
\end{bmatrix}.
\]
By a second use of the Borsuk-Ulam theorem, there exists a vector \(e_{n-1} \in \mathbb{S}^{n-1}\) orthogonal to \(e_{n}\) such that \(F_{n-1}(e_{n-1}) = 0\). Continuing in this manner, for all \(i = 2, \dots, n\), there exists a sequence of vectors \(e_i \in \mathbb{S}^{n-1}\), pairwise orthogonal, such that $$F_i(e_i) = 0.$$ 
We complete this family of vectors by \(e_1 = e_{n} \wedge \dots \wedge e_2\) so that it forms an orthonormal basis of \(\mathbb{R}^{n}\).

\end{proof}

Using this lemma, we can proceed to prove Theorem\ref{thmprncpalDomaine}.

\begin{proof}[Proof of Theorem \ref{thmprncpalDomaine}]
The construction above ensures that there exists an orthonormal basis \( (e_i) \) of \( \mathbb{R}^n \) such that for all \( i = 1, \dots, n \)

\begin{equation}\label{eq: fcttest}
  \mu_{i+1}(\Omega) \int_\Omega  \phi_i^2 \, dx \leq \int_\Omega |\nabla  \phi_i|^2 \, dx,  
\end{equation}
where \( \phi_i := g_{AB} \cdot e_i \). 

First, we develop inequality \eqref{eq: fcttest} by splitting the gradient of \( \phi_i \) on each of \( \Omega \cap H_A \) and \( \Omega \cap H_B \). We get the following expression for \( |\nabla \phi_i|^2 \)
\begin{equation}\label{gradient non simplifié}
|\nabla \phi_i|^2 =
\begin{cases} 
G'_{r_\Omega}(d_A(x))^2 \frac{|\overrightarrow{Ax} \cdot e_i|^2}{d_A(x)^2} +  \frac{G_{r_\Omega}(d_A(x))^2}{d_A(x)^2} \left( 1 - \frac{|\overrightarrow{Ax} \cdot e_i|^2}{d_A(x)^2} \right), & \text{if } x \in \Omega \cap H_A, \\
G'_{r_\Omega}(d_B(x))^2 \frac{|\overrightarrow{Bx} \cdot e_i|^2}{d_B(x)^2} + \frac{G_{r_\Omega}(d_B(x))^2}{d_B(x)^2} \left( 1 - \frac{|\overrightarrow{Bx} \cdot e_i|^2}{d_B(x)^2} \right), & \text{if } x \in \Omega \cap H_B. \\
\end{cases}
\end{equation}
For simplification, we adopt the following notation
$$x_i = \overrightarrow{Ax} \cdot e_i, \quad |x| = d_A(x), \quad y_i = \overrightarrow{Bx} \cdot e_i, \quad |y| = d_B(x), \quad \mu_i(\Omega) = \mu_i, \quad G_{r_\Omega} = G.$$
Thus, equation \eqref{gradient non simplifié} simplifies to
\begin{equation}\label{gradient simplifié}
|\nabla \phi_i|^2 =
\begin{cases} 
\frac{G(|x|)^2}{|x|^2} + \left( G'(|x|)^2 - \frac{G(|x|)^2}{|x|^2} \right) \frac{x_i^2}{|x|^2}, & \text{if } x \in \Omega \cap H_A, \\[10pt]
\frac{G(|y|)^2}{|y|^2} + \left( G'(|y|)^2 - \frac{G(|y|)^2}{|y|^2} \right) \frac{y_i^2}{|y|^2}, & \text{if } x \in \Omega \cap H_B.
\end{cases}
\end{equation}
Next, we divide inequality \eqref{eq: fcttest} by \( \mu_{i+1}(\Omega) \) and substitute the gradient expression from \eqref{gradient simplifié}. This leads to the following inequality for \( i = 1, \dots, n \)

\begin{align*}
\int_\Omega \phi_i^2 \, dx &\leq \frac{1}{\mu_{i+1}}  \int_{\Omega \cap H_A} \left( \frac{G(|x|)^2}{|x|^2} + \left(G'(|x|)^2 - \frac{G(|x|)^2}{|x|^2}\right) \frac{x_i^2}{|x|^2} \right) \, dx \\
&\quad + \frac{1}{\mu_{i+1}} \int_{\Omega \cap H_B} \left( \frac{G(|y|)^2}{|y|^2} + \left(G'(|y|)^2 - \frac{G(|y|)^2}{|y|^2}\right) \frac{y_i^2}{|y|^2} \right) \, dx.
\end{align*}
Summing from \( i = 1, \dots, n \), we obtain
\begin{align}
\sum_{i=1}^{n} \int_\Omega \phi_i^2 \, dx 
&\leq  \sum_{i=1}^{n} \frac{1}{\mu_{i+1}} \int_{\Omega \cap H_A} \left(  \frac{G(|x|)^2}{|x|^2} + \left(G'(|x|)^2 - \frac{G(|x|)^2}{|x|^2}\right) \frac{x_i^2}{|x|^2} \right) \, dx \label{premièreintégrale} \\
&\quad +  \sum_{i=1}^{n} \frac{1}{\mu_{i+1}} \int_{\Omega \cap H_B} \left(  \frac{G(|y|)^2}{|y|^2} + \left(G'(|y|)^2 - \frac{G(|y|)^2}{|y|^2}\right) \frac{y_i^2}{|y|^2} \right) \, dx. \label{deuxièmeintégrale}
\end{align}
Observe that the left-hand side \( \sum_{i=1}^{n} \int_\Omega \phi_i^2 \, dx \) is simply the same \( L^2 \) norm of the test function \( g_{AB} \) used by Bucur and Henrot. It consists of the sum of two Weinberger functions, one for each of the sets \( \Omega \cap H_A \) and \( \Omega \cap H_B \),

\[
\sum_{i=1}^{n} \int_\Omega \phi_i^2 \, dx = \int_{\Omega \cap H_A} G_{r_\Omega}^2(d_A(x)) \, dx + \int_{\Omega \cap H_B} G_{r_\Omega}^2(d_B(x)) \, dx.
\]

In order to estimate the right-hand side of inequality \eqref{premièreintégrale}, we apply the idea of Wang and Xia \cite{Xia2018OnAC} to each of these two integrals. First, note that
\[
\sum_{i=1}^{n} \frac{1}{\mu_{i+1}} \frac{x_i^2}{|x|^2} = \sum_{i=1}^{n-1} \frac{1}{\mu_{i+1}} \frac{x_i^2}{|x|^2} + \frac{1}{\mu_{n+1}} \frac{x_n^2}{|x|^2},
\]
and also that
\[
\frac{x_n^2}{|x|^2} = 1 - \sum_{i=1}^{n-1} \frac{x_i^2}{|x|^2}.
\]
Then it follows that
\[
\sum_{i=1}^{n} \frac{1}{\mu_{i+1}} \frac{x_i^2}{|x|^2} = \sum_{i=1}^{n-1} \frac{1}{\mu_{i+1}} \frac{x_i^2}{|x|^2} + \frac{1}{\mu_{n+1}} \left( 1 - \sum_{i=1}^{n-1} \frac{x_i^2}{|x|^2} \right).
\]
Substituting in the first integral over \( \Omega \cap H_A \), we get
\begin{equation}
\begin{split}
\sum_{i=1}^{n} \frac{1}{\mu_{i+1}} \int_{\Omega \cap H_A} &\left( G'(|x|)^2 - \frac{G(|x|)^2}{|x|^2} \right) \frac{x_i^2}{|x|^2} \, dx \\
&= \sum_{i=1}^{n-1} \frac{1}{\mu_{i+1}} \int_{\Omega \cap H_A} \left( G'(|x|)^2 - \frac{G(|x|)^2}{|x|^2} \right) \frac{x_i^2}{|x|^2} \, dx \\
&\quad + \frac{1}{\mu_{n+1}} \int_{\Omega \cap H_A} \left( G'(|x|)^2 - \frac{G(|x|)^2}{|x|^2} \right) \, dx \\
&\quad - \frac{1}{\mu_{n+1}} \int_{\Omega \cap H_A} \left( G'(|x|)^2 - \frac{G(|x|)^2}{|x|^2} \right) \sum_{i=1}^{n-1} \frac{x_i^2}{|x|^2} \, dx \\
&= \sum_{i=1}^{n-1} \int_{\Omega \cap H_A} \left( \frac{1}{\mu_{i+1}} - \frac{1}{\mu_{n+1}} \right) \left( G'(|x|)^2 - \frac{G(|x|)^2}{|x|^2} \right) \frac{x_i^2}{|x|^2} \, dx \\
&\quad + \frac{1}{\mu_{n+1}} \int_{\Omega \cap H_A} \left( G'(|x|)^2 - \frac{G(|x|)^2}{|x|^2} \right) \, dx.
\label{eq: égalité sur H_A}
\end{split}
\end{equation}
From the properties of Bessel functions, see Wang and Xia \cite{Xia2018OnAC}[Lemma 2.2] for instance, we have for that for all \( 0 < t \leq r_\Omega \),
$$ G'(t)^2 - \frac{G(t)^2}{t^2} \leq 0. $$
Therefore, this shows that
$$ \sum_{i=1}^{n-1} \int_{\Omega \cap H_A} \left( \frac{1}{\mu_{i+1}} - \frac{1}{\mu_{n+1}} \right) \left( G'(|x|)^2 - \frac{G(|x|)^2}{|x|^2} \right) \frac{x_i^2}{|x|^2} \, dx \leq 0. $$
Eliminating this negative term in inequality \eqref{eq: égalité sur H_A}, it becomes
\begin{align}
\sum_{i=1}^{n} \frac{1}{\mu_{i+1}} \int_{\Omega \cap H_A} \left( G'(|x|)^2 - \frac{G(|x|)^2}{|x|^2} \right) \frac{x_i^2}{|x|^2} \, dx 
&\leq \frac{1}{\mu_{n+1}} \int_{\Omega \cap H_A} \left( G'(|x|)^2 - \frac{G(|x|)^2}{|x|^2} \right) \, dx. \nonumber
\end{align}
Now, returning to the first term in inequality \eqref{premièreintégrale}, since
\begin{align}
    \sum_{i=1}^{n} \frac{1}{\mu_{i+1}} \int_{\Omega \cap H_A} \frac{G(|x|)^2}{|x|^2} \, dx 
    &+ \frac{1}{\mu_{n+1}} \int_{\Omega \cap H_A} \left( G'(|x|)^2 - \frac{G(|x|)^2}{|x|^2} \right) \, dx \notag \\
    &= \frac{1}{\mu_{n+1}} \int_{\Omega \cap H_A} G'(|x|)^2 \, dx \notag \\
    &+ \sum_{i=1}^{n-1} \frac{1}{\mu_{i+1}} \int_{\Omega \cap H_A} \frac{G(|x|)^2}{|x|^2} \, dx,
\end{align}

we obtain that
\begin{align}
    \sum_{i=1}^{n} \frac{1}{\mu_{i+1}} \int_{\Omega \cap H_A} &\left( \frac{G(|x|)^2}{|x|^2} + \left(G'(|x|)^2 - \frac{G(|x|)^2}{|x|^2}\right) \frac{x_i^2}{|x|^2} \right) \, dx \notag \\
    &\leq \frac{1}{n-1} \sum_{i=1}^{n-1} \frac{1}{\mu_{i+1}} \int_{\Omega \cap H_A} \left( G'(|x|)^2 + (n-1) \frac{G(|x|)^2}{|x|^2} \right) \, dx. 
\end{align}

By symmetry with respect to the point \( B \), the second term on the right-hand side of inequality \eqref{deuxièmeintégrale} yields the same estimates on \( \Omega \cap H_B \),
\begin{align}
    \sum_{i=1}^{n} \frac{1}{\mu_{i+1}} \int_{\Omega \cap H_B} &\left( \frac{G(|y|)^2}{|y|^2} + \left(G'(|y|)^2 - \frac{G(|y|)^2}{|y|^2}\right) \frac{y_i^2}{|y|^2} \right) \, dx \notag \\
    &\leq \frac{1}{n-1} \sum_{i=1}^{n-1} \frac{1}{\mu_{i+1}} \int_{\Omega \cap H_B} \left( G'(|y|)^2 + (n-1) \frac{G(|y|)^2}{|y|^2} \right) \, dx.
\end{align}

Hence, replacing these estimates in \eqref{premièreintégrale} and \eqref{deuxièmeintégrale} gives:
\begin{align}\label{inégalité finale avant deplacement masse}
    &\frac{\int_{\Omega \cap H_A} G(|x|)^2 \, dx 
    + \int_{\Omega \cap H_B} G(|y|)^2 \, dx}
    { \int_{\Omega \cap H_A} \left( G'(|x|)^2 + (n-1) \frac{G(|x|)^2}{|x|^2} \right) \, dx 
    + \int_{\Omega \cap H_B} \left( G'(|y|)^2 + (n-1) \frac{G(|y|)^2}{|y|^2} \right) \, dx } \notag \\
    &\leq \frac{1}{n-1} \sum_{i=1}^{n-1} \frac{1}{\mu_{i+1}}.
\end{align}

It remains to estimate the left-hand side in inequality \eqref{inégalité finale avant deplacement masse}, and this is done using the mass displacement argument, as we explained in the previous section through the work of Bucur and Henrot \cite{BucurHenrot}. We then have
\begin{align}\label{ineq: strict}
    &\frac{
    \int_{\Omega \cap H_A} G(|x|)^2 \, dx + \int_{\Omega \cap H_B} G(|y|)^2 \, dx
    }{
    \int_{\Omega \cap H_A} \left( G'(|x|)^2 + (n-1) \frac{G(|x|)^2}{|x|^2} \right) \, dx 
    + \int_{\Omega \cap H_B} \left( G'(|y|)^2 + (n-1) \frac{G(|y|)^2}{|y|^2} \right) \, dx 
    } \notag \\
    &\geq \frac{
    \int_{B_{r_\Omega}} G(|x|)^2 \, dx 
    }{
    \int_{B_{r_\Omega}} \left( G'(|x|)^2 + (n-1) G(|x|)^2 \right) \, dx
    } \notag \\
    &= \frac{1}{\mu_1(B_{r_\Omega})},
\end{align}
where we recall that $B_{r_\Omega}$ is a ball centred at the origin having a volume of $\frac{|\O|}{2}$.

Since \( \mu_1(B_{r_\Omega}) = \mu_2(\B^n \sqcup \B^n) \frac{|\B^n \sqcup \B^n|^{2/n}}{|\Omega|^{2/n}} \), we finally deduce that
\begin{equation}
    \sum_{i=1}^{n-1} \frac{1}{\mu_{i+1}} \geq \frac{(n-1)|\Omega|^{2/n}}{\mu_2(\B^n \sqcup \B^n) |\B^n \sqcup \B^n|^{2/n}}.
\end{equation}

To conclude the proof of Theorem \ref{thmprncpalDomaine}, we analyze the case of equality. Any displacement of mass on a set of positive measure would make inequality \eqref{ineq: strict} strict. Therefore, equality can only hold if both \( H_A \) and \( H_B \) are balls, each with mass \(\frac{|\Omega|}{2}\), except possibly on a set of measure zero. Hence, when equality is achieved, \( \Omega \) is almost everywhere equivalent to the union of two disjoint balls, each having mass \(\frac{|\Omega|}{2}\). This completes the proof of Theorem \ref{thmprncpalDomaine}.

\end{proof}

\section{Estimates on the sum of the reciprocals of eigenvalues for spheres and surfaces}

In this section, we focus on the case of spheres and surfaces. First, we recall the Hersch-Szeg\H{o} center of mass argument, and then we proceed to discuss the corresponding folding method relevant to this situation. The general topological framework is similar to the case of domains established in the previous section.

Let \( \mathbb{S}^n \) be the unit sphere in \( \mathbb{R}^{n+1} \) equipped with a metric \( g \) that is conformal to the standard metric \( g_0 \). We use the standard variational characterization of \( \lambda_k(\mathbb{S}^n, g) \)
\begin{gather}\label{caractérisation variationnelle Sphère}
\lambda_k(\mathbb{S}^n, g) = \min_{f \in A_k \setminus \{0\}} \frac{\int_{\mathbb{S}^n} |\nabla f|^2 \, dv_g}{\int_{\mathbb{S}^n} f^2 \, dv_g},
\end{gather}
where \( A_k \) is the following subspace of the Sobolev space \( H^1(\mathbb{S}^n, g) \),
\[
A_k = \left\{ f \in H^1(\mathbb{S}^n, g) \mid \int_{\mathbb{S}^n} f f_j \, dv_g = 0 \text{ for } j = 0, 1, \dots, k-1 \right\}.
\]
Here, \( (f_j)_{j \geq 0} \) is an orthonormal basis of \( L^2(\mathbb{S}^n, g) \) corresponding to the eigenvalues \( \lambda_j(\mathbb{S}^n, g) \).

\subsection{Hersch-Szeg\H{o} center of mass of a Borel measure on a sphere}

For any \( v \in \mathbb{S}^n \), define the coordinate functions \( X_v(x) = \langle v, x \rangle \)  on the sphere \( \mathbb{S}^n \). These functions form the corresponding eigenspace to the eigenvalue \( \lambda_1(\mathbb{S}^n, g_0) = n \).

For each \( \xi \in \mathbb{B}^{n+1} \subset \mathbb{R}^{n+1} \), define the map \( \phi_\xi: \mathbb{S}^n \to \mathbb{S}^n \) as follows:
\begin{equation}\label{defiphi_xi}
    \phi_{\xi}(x) = \xi + \frac{1 - |\xi|^2}{|x + \xi|^2}(x + \xi).
\end{equation}
The family of transformations \( \phi_\xi \) are conformal automorphisms of \( \mathbb{S}^n \) and are particularly useful for renormalizing the center of mass of a given Borel measure, it has been widely used in the literature for deriving various upper bounds for eigenvalues of a given spectral problem. A recent version of this result can be found in Laugesen~\cite[Corollary 5]{Laugesencenterofmass}.

\begin{lemma}[Hersch-Szeg\H{o} center of mass]\label{lemmeHersch}
Let \( \mu \) be a Borel measure on the unit sphere \( \mathbb{S}^n \), with \( n \geq 1 \), satisfying \( 0 < \mu(\mathbb{S}^n) < \infty \). 
If for all \( y \in \mathbb{S}^n \),
\begin{gather}\label{ineq:halfmass}
    \mu(\{y\}) < \frac{1}{2}\mu(\mathbb{S}^n),
\end{gather}
then a unique point \( \xi = \xi(\mu) \in \mathbb{B}^{n+1} \) exists such that:
\begin{equation}\label{mesurenonatomique}
    \int_{\mathbb{S}^n} \phi_{\xi} \, d\mu = 0.
\end{equation}
This point \( \xi(\mu) \) depends continuously on the measure \( \mu \). That is, if \( \mu \) satisfies~\eqref{ineq:halfmass} and if \( \mu_k \to \mu \) weakly, where each \( \mu_k \) also satisfies~\eqref{ineq:halfmass}, then \( \xi(\mu_k) \to \xi(\mu) \) as \( k \to \infty \).
\end{lemma}
The point \( \xi \) is called the \emph{center of mass} of the measure \( d\mu \).
\begin{rem}
    Equation \eqref{mesurenonatomique} is interpreted in the vectorial sense, and it is equivalent to write:
\begin{equation}
    \int_{\mathbb{S}^n} X_v \circ \phi_{\xi} \, d\mu = 0,
\end{equation}
for all \( v \in \mathbb{S}^n \).
\end{rem}
This version of center of mass argument allows us to adapt the topological argument of Lemma \ref{Topological arguments BU domains} to our current setting.
\begin{lemma}
Let \( g \) be a metric that is conformal to the standard one \( g_0 \) on \( \mathbb{S}^n \). Then, there exists a point \( \xi \in \mathbb{B}^{n+1} \) and an orthonormal basis \( (e_i) \) of \( \mathbb{R}^{n+1} \) such that the coordinate functions \( X_{e_i} \circ \phi_{\xi} \) satisfy
\[
\lambda_i(\mathbb{S}^n, g) \int_{\mathbb{S}^n} X_{e_i}^2 \circ \phi_{\xi} \, dv_g 
\leq 
\int_{\mathbb{S}^n} |\nabla_g X_{e_i} \circ \phi_{\xi} |^2 \, dv_g,
\]
for all \( i = 1, \dots, n+1 \).
\end{lemma}

\begin{proof}
First, set \( \xi \in \mathbb{B}^{n+1} \) to be the center of mass of the volume measure \( dv_g \). By definition of the center of mass, this implies that for all \( v \in \mathbb{S}^n \),
\begin{equation}
    \int_{\mathbb{S}^n} X_v \circ \phi_{\xi} \, dv_g = 0.
\end{equation}
Next, define the map \( F_{n+1} : \mathbb{S}^n \to \mathbb{R}^n \) as
\[
F_{n+1}(p) =
\begin{bmatrix} 
\int_{\mathbb{S}^n} X_p \circ \phi_{\xi} f_1 \, dv_g \\
\vdots \\
\int_{\mathbb{S}^n} X_p \circ \phi_{\xi} f_n \, dv_g
\end{bmatrix}.
\]
Using the Borsuk-Ulam theorem recursively, as in Lemma \ref{Topological arguments BU domains}, we deduce that for all \( i = 2, \dots, n+1 \), there exists a sequence of vectors \( e_i \in \mathbb{S}^n \), pairwise orthogonal, such that \( F_i(e_i) = 0 \). 
We complete this family of vectors by defining \( e_1 = e_{n+1} \wedge \dots \wedge e_2 \), ensuring it forms an orthonormal basis of \( \mathbb{R}^{n+1} \).

Finally, applying the variational characterization of the eigenvalues \eqref{caractérisation variationnelle Sphère}, we obtain, for all \( i = 1, \dots, n+1 \)
\[
\lambda_i(\mathbb{S}^n, g) \int_{\mathbb{S}^n} X_{e_i}^2 \circ \phi_{\xi} \, dv_g 
\leq 
\int_{\mathbb{S}^n} |\nabla_g X_{e_i} \circ \phi_{\xi} |^2 \, dv_g.
\]
\end{proof}
We are now in a position to prove Theorem \ref{thmprincSphère1}.
\begin{proof}[Proof of Theorem \ref{thmprincSphère1}]
Starting from the inequality established in the previous lemma, we invert and sum over \( i = 1, \dots, n+1 \), which yields
\begin{equation}\label{equationTchebychev}
    \text{vol}(\mathbb{S}^n, g) 
    \leq 
    \sum_{i=1}^{n+1}\frac{1}{\lambda_i(\mathbb{S}^n, g)} 
    \int_{\mathbb{S}^n} |\nabla_g X_{e_i}|^2 \, dv_g.
\end{equation}
Applying Hölder's inequality to the terms in \eqref{equationTchebychev}, we obtain
\begin{align}\label{equationfinaleapresholder}
    \text{vol}(\mathbb{S}^n, g) 
    \leq 
    \sum_{i=1}^{n+1} \frac{1}{\lambda_i(\mathbb{S}^n, g)} 
    \left( \int_{\mathbb{S}^n} |\nabla_g X_{e_i}|^n \, dv_g \right)^{2/n} 
    \text{vol}(\mathbb{S}^n, g)^{1-2/n}.
\end{align}
We observe that the term 
\[
\left( \int_{\mathbb{S}^n} |\nabla_g X_{e_i}|^n \, dv_g \right)^{2/n}
\]
is a conformal invariant on \( \mathbb{S}^n \). Specifically, when \( g \in [g_0] \), this term simplifies to
\[
\left( \int_{\mathbb{S}^n} |\nabla_g X_{e_i}|^n \, dv_g \right)^{2/n} = \frac{n}{n+1} K_n w_n^{2/n},
\]
where \( K_n \) is the constant defined in \eqref{defi de Kn}. The explicit computation of \( K_n \) was performed by Girouard, Nadirashvili, and Polterovich in \cite{GNP}.

Substituting this expression into inequality \eqref{equationfinaleapresholder}, we deduce 
\[
\text{vol}(\mathbb{S}^n, g) 
\leq 
\sum_{i=1}^{n+1} \frac{1}{\lambda_i(\mathbb{S}^n, g)} \frac{n}{n+1} K_n w_n^{2/n} 
\text{vol}(\mathbb{S}^n, g)^{1-2/n}.
\]

Simplifying further and isolating the reciprocal sum of eigenvalues, we find
\begin{equation*}
    \frac{1}{\lambda_1(\mathbb{S}^n, g)} + \cdots + \frac{1}{\lambda_{n+1}(\mathbb{S}^n, g)} 
    \geq 
    \frac{(n+1) \, \text{vol}(\mathbb{S}^n, g)^{2/n}}{n K_n w_n^{2/n}},
\end{equation*}
thereby completing the proof of Theorem \ref{thmprincSphère1}.
\end{proof}

\subsection{Folding onto spherical caps}
The folding method described in the previous section was first introduced by Nadirashvili \cite{nadirashvili2002isoperimetric} for the sphere \(\mathbb{S}^2\) and later generalized by Girouard, Nadirashvili, and Polterovich \cite{GNP} to higher-dimensional spheres. This approach involves folding measures onto spherical caps and leveraging a topological argument to identify a cap that simultaneously satisfies two orthogonality conditions. We briefly outline the construction of this method and apply it, as in Lemma \ref{Topological arguments BU domains}, to prove Theorem \ref{thmprincsphere2}.

\smallskip

Let $(p,t)\in \Sn \times (-1,1)$ and consider
the hemisphere 
$$C_{(p,0)}:=\{x \in \Sn\,:\, \langle x, p\rangle >0 \}$$ 
centered on $p\in\Sn$. The spherical cap $C_{(p,t)}$ centered at $p$ and with radius $t$ is defined as
$$C_{(p,t)} = \phi_{-tp}(C_{(p,0)}).$$
Thus, the space of spherical caps is naturally identified with $(-1,1)\times\Sn$, and its natural compactification is $\overline{\mathbb{B}}^{n+1}$, where the origin $0$ corresponds to the limit as $t\to 1$.

Given \( p \in \mathbb{S}^n \), the reflection \( R_p : \mathbb{S}^n \to \mathbb{S}^n \) across the hyperplane perpendicular to \( p \) and containing the origin is given by
$$ R_p(x) = x - 2 \langle x, p \rangle p. $$

Conjugation with the automorphism defining a spherical cap \( C \) allows the definition of reflections \( \tau_C : \mathbb{S}^n \to \mathbb{S}^n \) across the boundary of any spherical cap \( C = C_{(p, t)} \):
$$ \tau_C := \phi_{-tp} \circ R_p \circ \phi_{tp}. $$

For the spherical cap \( C = C_{(p, t)} \),  define the folding map \( F_C = F_{(p, t)} : \mathbb{S}^n \to C \) by
\begin{equation}\label{foldingmap}
F_C(x) =
\begin{cases}
    x & \text{if } x \in C, \\
    \tau_C(x) & \text{otherwise}.
\end{cases}
\end{equation}
\smallskip
By a standard symmetry argument, originating from Girouard, Nadirashvili, and Polterovich \cite{GNP}, and later refined by Petrides \cite{petrides2014maximization}, Freitas and Laugesen \cite{F-L1}, and Kim \cite{kim2022maximization}, there exists a cap \( C \subset \mathbb{S}^n \) and an associated center of mass \( \xi_C \in \mathbb{B}^{n+1} \) such that, for all \( v \in \mathbb{S}^n \), the following orthogonality conditions hold
\begin{gather}
    \int_{\mathbb{S}^n} X_v \circ \phi_{\xi_C} \circ F_C \, dv_g = 0, \label{eq:1ère_orthogo_Sphère} \\
    \int_{\mathbb{S}^n} X_v \circ \phi_{\xi_C} \circ F_C \cdot f_1 \, dv_g = 0. \label{eq:2eme_orthogo_Sphère}
\end{gather}

\begin{proof}[Proof of Theorem \ref{thmprincsphere2}]
Consider $C,\xi_C$ as in \eqref{eq:1ère_orthogo_Sphère} and \eqref{eq:2eme_orthogo_Sphère}, and then apply recursively the Borsuk-Ulam theorem to the maps
    \( \Tilde{F}_{n+1} : \mathbb{S}^n \to \mathbb{R}^n \) defined as
\[
\Tilde{F}_{n+1}(p) =
\begin{bmatrix} 
\int_{\mathbb{S}^n} X_p \circ \phi_{\xi_C} \circ F_C f_2 \, dv_g \\
\vdots \\
\int_{\mathbb{S}^n} X_p \circ \phi_{\xi_C} \circ F_C f_{n+1} \, dv_g
\end{bmatrix}.
\]
Using the variational characterization \eqref{caractérisation variationnelle Sphère} of $\lambda_{i+1}(\S^n,g)$, with $i=1 \cdots n+1$ , we obtain the following inequalities. 
\[
\lambda_{i+1}(\mathbb{S}^n, g) \int_{\mathbb{S}^n} X_{e_i}^2 \circ \phi_{\xi_C} \circ F_C  \, dv_g \leq \int_{\mathbb{S}^n} |\nabla_g X_{e_i }\circ \phi_{\xi_C} \circ F_C  |^2 \, dv_g,
\]
for all \(i = 1, \dots, n+1\).

 Inverting this inequality and summing, we get
\begin{align}
 \text{vol}(\mathbb{S}^n, g) &\leq \sum_{i=2}^{n+2}\frac{1}{\lambda_i(\mathbb{S}^n, g)} \int_{\mathbb{S}^n} |\nabla_g X_{e_i} \circ \phi_{\xi_C} \circ F_C|^2 \, dv_g \\
&<  \sum_{i=2}^{n+2}\frac{2^{2/n}}{\lambda_i(\mathbb{S}^n, g)}  \left( \int_{\mathbb{S}^n} |\nabla_g X_{e_i}|^n\, dv_g \right)^{2/n} \text{vol}(\mathbb{S}^n, g)^{1-2/n} \\
&= \sum_{i=2}^{n+2}\frac{1}{\lambda_i(\mathbb{S}^n, g)}  \frac{n}{n+1} 2^{2/n} K_n w_n^{2/n} \text{vol}(\mathbb{S}^n, g)^{1-2/n}
\end{align}
Thus, after reorganizing the terms, we deduce
\begin{equation*}
     \frac{1}{\lambda_2(\mathbb{S}^n, g)}+ \cdots + \frac{1}{\lambda_{n+2}(\mathbb{S}^n, g)} > \frac{(n+1)\text{vol}(\S^n,g)}{2^{2/n} n K_n w_n^\frac{2}{n}},
     \end{equation*}
and the proof is finished.

\end{proof}

\subsection{The case of closed orientable surfaces and proof of Theorem \ref{thmprincsurfaces}}

When dealing with a closed, orientable surface \((M, g)\) of genus \(\gamma\), one can approach the problem as if it were a sphere by considering a branched covering
$$ \phi : M \to \mathbb{S}^2 $$
of degree
$$ d \leq \left\lfloor \frac{\gamma+3}{2} \right\rfloor. $$

The existence of such a map follows from the Riemann-Roch theorem, and the bound on the degree is due to El Soufi and Ilias \cite{ElSoufiIlias1984}, who improved an upper bound on the first eigenvalue originally established by Yang and Yau \cite{YangYau}.
\smallskip

Let \(f_i\) be the eigenfunctions associated with the eigenvalues \(\lambda_i(M, g)\). By considering the push forward volume measure of $(M,g)$ onto $\S^2$ by $\phi$, making the same topological arguments as the case of higher-dimensional spheres in \eqref{eq:1ère_orthogo_Sphère} and \eqref{eq:2eme_orthogo_Sphère}, one can always assume that there exists a spherical cap \(C \subset \S^2\) and a center of mass \(\xi_C\) such that for all \(v \in \mathbb{S}^2\):
\begin{gather}
    \int_{M} X_v \circ \phi_{\xi_C} \circ F_C \circ \phi \, dv_g = 0,  \\
    \int_{M} X_v \circ \phi_{\xi_C} \circ F_C \circ \phi \cdot f_1 \, dv_g = 0. 
\end{gather}
where \(\phi_\xi\) be the family of conformal transformations on \(\mathbb{S}^2\) defined in \eqref{defiphi_xi}, where $F_C$ is the folding map defined in \eqref{foldingmap} and where $X_v $ are the coordinate functions on $\S^2$.

\begin{proof}[Proof of Theorem \ref{thmprincsurfaces}]

Through the construction made in the proof of Theorem \ref{thmprincsphere2}, we get the following inequalities

\[
\begin{aligned}
\lambda_4 \int_M X^2_{e_3} \circ \phi_{\xi_C} \circ F_C\circ \phi \, dv_g &\leq \int_M |\nabla X_{e_3} \circ \phi_{\xi_C} \circ F_C \circ \phi|^2 \, dv_g, \\
\lambda_3 \int_M X^2_{e_2} \circ \phi_{\xi_C} \circ F_C \circ \phi \, dv_g &\leq \int_M |\nabla X_{e_2} \circ \phi_{\xi_C} \circ F_C|^2 \, dv_g, \\
\lambda_2 \int_M X^2_{e_1} \circ \phi_{\xi_C} \circ F_C \circ \phi \, dv_g &\leq \int_M |\nabla X_{e_1} \circ \phi_{\xi_C} \circ F_C \circ \phi|^2 \, dv_g.
\end{aligned}
\]
By conformal invariance in dimension $2$, a classical computation shows that for all \(i = 1, \dots, 3\),
\begin{align*}  
 \int_{M} |\nabla X_{e_i} \circ \phi_{\xi_C} \circ F_C \circ \phi|^2 \, dv_g &<  2\int_{\S^2} |\nabla X_{e_i}|^2 \, dv_{g_0} \\
 &=\frac{16\pi d}{3}.
 \end{align*}
Thus, inverting and summing these inequalities, we get
$$ \frac{1}{\lambda_4} + \frac{1}{\lambda_3} + \frac{1}{\lambda_2} > \frac{3\text{Area}(g)}{16 \pi d}.$$
Finally, the result follows from the upper bound for the degree  \(d\) of the branched covering map $\phi$.

\end{proof}

\bibliographystyle{plain}
\bibliography{Bibliographie}
\end{document}